\documentclass[12pt]{amsart}
\usepackage{fullpage, amsfonts,amsthm,stmaryrd,url,color}
\usepackage[all]{xy}

\newtheorem{theorem}{Theorem}[section]
\newtheorem{lemma}[theorem]{Lemma}
\newtheorem{lemma-def}[theorem]{Lemma-Definition}
\newtheorem{prop}[theorem]{Proposition}

\newtheorem{question}[theorem]{Question}

\theoremstyle{definition}
\newtheorem{defn}[theorem]{Definition}
\newtheorem{remark}[theorem]{Remark}
\newtheorem{example}[theorem]{Example}

\newcommand{\CC}{\mathbb{C}}
\newcommand{\FF}{\mathbb{F}}
\newcommand{\QQ}{\mathbb{Q}}
\newcommand{\AAA}{\mathbb{A}}

\newcommand{\ZZ}{\mathbb{Z}}
\newcommand{\calH}{\mathcal{H}}

\newcommand{\calA}{\mathcal{A}}

\newcommand{\frako}{\mathfrak{o}}

\DeclareMathOperator{\charac}{char}

\DeclareMathOperator{\trdeg}{trdeg}

\begin{document}

\title{Endomorphisms of power series fields and residue fields of Fargues-Fontaine curves}
\author{Kiran S. Kedlaya and Michael Temkin}
\date{March 18, 2017}
\thanks{Thanks to Brian Lawrence, Michel Matignon, Andrew Obus, and Tom Scanlon for helpful discussions.
Some of this work was carried out during the MSRI fall 2014 semester program ``New geometric methods in number theory and automorphic forms'' supported by NSF grant DMS-0932078. Kedlaya received additional support from NSF grants DMS-1101343 and
DMS-1501214 and from UC San Diego
(Stefan E. Warschawski Professorship). Temkin was supported by the Israel Science Foundation (grant No. 1159/15).}

\maketitle

\begin{abstract}
We show that for $k$ a perfect field of characteristic $p$,
there exist endomorphisms of the completed algebraic closure of $k((t))$ which are not bijective. As a corollary, we resolve a question of Fargues and Fontaine by showing that for $p$ a prime and $\CC_p$ a completed algebraic closure of $\QQ_p$, there exist closed points of the Fargues-Fontaine curve associated to $\CC_p$ whose residue fields are not (even abstractly) isomorphic to $\CC_p$ as topological fields.
\end{abstract}

\section{Introduction}

In this short note, we address the following question.
By an \emph{analytic field}, we will always mean a field complete with respect to a nonarchimedean multiplicative absolute value (assumed to be real-valued and written multiplicatively); by default, we always allow the trivial absolute value.

\begin{question} \label{Q:induced isomorphism}
Let $K$ be an analytic field. Let $k$ be a trivially valued subfield of $K$. Is every continuous $k$-linear homomorphism from $K$ to itself which induces automorphisms of residue fields and value groups necessarily surjective (and hence an automorphism)?
\end{question}

We will view Question~\ref{Q:induced isomorphism} as a collection of distinct cases indexed by the choice of $K,k$.
For example, one has affirmative answers in the following cases:
\begin{itemize}
\item
when $K$ is trivially valued, discretely valued, or more generally spherically complete (Proposition~\ref{P:spherically complete isomorphism});
\item
when $\charac(k) = 0$ and $K$ is the completed algebraic closure of a power series field over $k$
(Remark~\ref{R:char 0});
\end{itemize}
whereas one has negative answers in the following cases:
\begin{itemize}
\item
in certain cases in characteristic $0$ (Example~\ref{exa:no induced isomorphism});
\item
when $\charac(k) > 0$ and $K$ is the completed perfect closure of a power series field over $k$
(see \cite{kedlaya-series}).
\end{itemize}

Hereafter, fix a prime number $p$.
Our main result is a negative answer to Question~\ref{Q:induced isomorphism}
when $\charac(k)  = p$ and $K$ is the completed algebraic closure of a power series field over $k$.
\begin{theorem} \label{T:induced isomorphism}
Let $K$ be a completed algebraic closure of $k((t))$ for some field $k$ of characteristic $p$. Then there exists a continuous $k$-linear homomorphism $\tau: K \to K$ which is not an isomorphism.
\end{theorem}

The proof depends on a calculation using completed modules of K\"ahler differentials of analytic fields, as recently studied by the second author \cite{temkin}. We develop here the bare minimum of this subject needed for the proof of Theorem~\ref{T:induced isomorphism}; a more detailed treatment of completed differentials between analytic fields will be given by the second author elsewhere.

Theorem~\ref{T:induced isomorphism} was prompted by an application to a foundational question of $p$-adic Hodge theory, specifically in the \emph{perfectoid correspondence} (commonly known as \emph{tilting}) between nonarchimedean fields in mixed and equal characteristics
(generalizing the \emph{field of norms correspondence} of Fontaine and Wintenberger).
A nonarchimedean field $K$ of residue characteristic $p$ is \emph{perfectoid} if it is not discretely valued and the Frobenius automorphism on $\frako_K/(p)$ is surjective.
Given such a field, let $K^\flat$ be the inverse limit of $K$ under the $p$-power map; one then shows that $K^\flat$ naturally carries the structure of a perfectoid (and hence perfect) nonarchimedean field of equal characteristic $p$ and that there is a canonical isomorphism between the absolute Galois groups of $K$ and $K^\flat$
\cite{kedlaya-new-phigamma, kedlaya-liu1, scholze1}.
The functor $K \mapsto K^{\flat}$ is not fully faithful, even on fields of characteristic $0$; for instance, one can construct many algebraic extensions of $\QQ_p$ whose completions $K$ map to the completed perfect closure of a power series field over $\FF_p$ (e.g., the cyclotomic extension $\QQ_p(\mu_{p^\infty})$ and the Kummer extension $\QQ_p(p^{1/p^\infty})$).
However, Fargues and Fontaine have asked \cite[Remark~2.24]{fargues-fontaine-durham}
(see also \cite{fargues-fontaine}) whether this can happen for a completed algebraic closure of $\QQ_p$, and using Theorem~\ref{T:induced isomorphism} we are able to answer this question.

\begin{theorem} \label{T:same untilt Cp}
Let $\CC_p$ be a completed algebraic closure of $\QQ_p$.
Then there exists a perfectoid field $K$ which is not isomorphic to $\CC_p$ as a topological field, but for which there exists an isomorphism $K^\flat \cong \CC_p^\flat$.
\end{theorem}

This result admits the following geometric interpretation. For each perfectoid field $K$,
Fargues and Fontaine define an associated scheme $X_K$ which is a ``complete curve'' (i.e., a regular one-dimensional noetherian scheme equipped with a surjection of its Picard group onto $\ZZ$) in terms of which $p$-adic Hodge theory over $K$ can be simply formulated.
Theorem~\ref{T:same untilt Cp} implies that for $K = \CC_p$, there exists a closed point of $X_K$ whose residue field is not isomorphic to $\CC_p$.

We conclude this introduction by pointing out that after we prepared our proof of Theorem~\ref{T:induced isomorphism}, we learned that this statement is a special case of a result of Matignon and Reversat \cite[Th\'eor\`eme~2]{matignon-reversat}. However, since our proof of the special case is somewhat simpler than the more general argument of Matignon--Reversat, we have elected to retain the proof here.

\section{Analytic fields and completed differentials}
\label{sec:covers}

As a technical input into the proof of Theorem~\ref{T:induced isomorphism}, we
review some basic properties of analytic fields and completed differentials.

\begin{defn}
By an \emph{analytic field}, we will mean a field equipped with a multiplicative nonarchimedean absolute value with respect to which the field is complete. By default, we allow the trivial absolute value. When we consider an \emph{extension} $L/K$ of analytic fields, we require that the absolute value on $L$ restricts to the absolute value on $K$.
\end{defn}

\begin{defn}
We say that an extension $L/K$ of analytic fields is \emph{primitive} if there exists $t \in L^\times$ such that $K(t)$ is dense in $L$; we will write $L = \widehat{K(t)}$ if we need to indicate the choice of $t$.

With $t$ given, the extension $\widehat{K(t)}/K$ corresponds to a point in the projective line over $K$ in the category of Berkovich nonarchimedean analytic spaces
\cite{berkovich1}. Without $t$ given, the points associated to $L/K$ are all of the same type 1--4 in Berkovich's classification \cite[(1.4.4)]{berkovich1}; we thus classify $L/K$ accordingly.
Write
\[
E_{L/K} = \dim_\QQ (|L^\times|/|K^\times|) \otimes_{\ZZ} \QQ, \qquad
F_{L/K} = \trdeg_{\kappa(K)} \kappa(L),
\]
where $\kappa(*)$ denotes the residue field of $*$;
these are determined by the type of $L/K$ as follows.
\begin{center}
\begin{tabular}{c|ccc}
Type of $L/K$ & $E_{L/K}$ & $F_{L/K}$ \\
\hline
1 & 0 & 0 \\
2 & 0 & 1 \\
3 & 1 & 0 \\
4 & 0 & 0
\end{tabular}
\end{center}
In all cases we have $E_{L/K} + F_{L/K} \leq 1$, as per Abhyankar's inequality (e.g., see \cite[Lemma~2.1.2]{temst}).
However, types 1 and 4 cannot be distinguished using $E_{L/K}$ and $F_{L/K}$ alone:
one must instead observe that
$L/K$ is of type 1 if and only if $L$ embeds into the completed algebraic closure of $K$.
\end{defn}

In order to better distinguish between primitive extensions of types 1 and 4,
we will use completed modules of differentials.

\begin{defn}
Let $L/K$ be an extension of analytic fields. As described in \cite[\S 4]{temkin},
the module $\Omega_{L/K}$ admits a maximal seminorm $\left\| \bullet \right\|$ (the \emph{K\"ahler seminorm}) with respect to which $d_{L/K}: L \to \Omega_{L/K}$ is nonexpanding. Let $\widehat{\Omega}_{L/K}$ denote the completion of $\Omega_{L/K}$
with respect to $\left\| \bullet \right\|$; it receives an induced derivation
$\widehat{d}_{L/K}: L \to \widehat{\Omega}_{L/K}$.
\end{defn}

\begin{lemma} \label{L:primitive}
Let $L/K$ be a primitive extension, and choose $t \in L$ such that $K(t)$ is dense in $L$.
\begin{enumerate}
\item[(a)]
The module $\widehat{\Omega}_{L/K}$ is generated over $L$ by the single element $\widehat{d}_{L/K}(t)$.
\item[(b)]
The equality $\widehat{\Omega}_{L/K} = 0$ holds if and only if the separable closure of $K$ in $L$ is dense. (Note that this condition implies that $L/K$ is of type 1, and conversely whenever $\charac(K)=0$.)
\end{enumerate}
\end{lemma}
\begin{proof}
Since $\Omega_{K(t)/K}$ is generated by $d_{L/K}(t)$, (a) is obvious.

Let $l$ be the separable integral closure of $K$ in $L$.
If $l$ is dense in $L$ (which forces $L/K$ to be of type 1), then $\Omega_{l/K} = 0$
and so $\widehat{\Omega}_{L/K} = 0$. This proves the inverse implication in (b).

Suppose that $L/K$ is not of type 1.
Let $K'$ be a completed algebraic closure of $K$
and put $L' = l \widehat{\otimes}_K K'$; then the natural map
$\Omega_{L/K} \widehat{\otimes}_L L' \to \Omega_{L'/K'}$ sends
$\widehat{d}_{L/K}(t) \otimes 1$ to $\widehat{d}_{L'/K'}(t)$.
The latter is nonzero by \cite[Theorem~2.3.2(i)]{cohen-temkin-trushin},
so $\widehat{d}_{L/K}(t) \neq 0$.

It remains to consider the case when $L/K$ is of type 1 but $l$ is not dense in $L$. (Note that this last step
is not needed for the proof of Theorem~\ref{T:induced isomorphism}, so the uninterested reader can skip it.)
Observe that any separable extension of $\widehat{l}$ is the closure of a separable extension of $l$. Since
$l$ is separably closed in $L$, we obtain that $\widehat{l}$ is separably closed in $L$ too.
It suffices to show that $\widehat{d}_{L/\widehat{l}}(t)\neq 0$, so after replacing $K$ by $\widehat{l}$ we can assume
that $K=l$.

Fix an embedding of $L$ into $K'$.
Let $G$ be the group of continuous automorphisms of $K'$ fixing $K$; this group is naturally identified with the absolute Galois group of $K$. The subgroup $H$ fixing $L$ is closed in $G$, and hence is the absolute Galois group of some separable extension $L_0$ of $K$. If $\charac K = 0$, then the Ax-Sen theorem
\cite{ax} applied to both $L_0$ and $L$ implies that $(K')^H = \widehat{L_0} = L$, but this contradicts our previous assumption that $K = l \neq L$. We must then have $\charac K = p> 0$. By Ax-Sen again,
we have $(K')^H = \widehat{L_0^{1/p^\infty}} = \widehat{L^{1/p^\infty}}$.
If $L_0 \neq K$, we may choose a separable irreducible polynomial $P \in K[T]$ of degree $>1$ with a root in $L_0$; by Krasner's lemma, $P$ has a root $x$ in $L^{1/p^n}$ for some sufficiently large $n$. But then $x^{p^n} \in L$ generates a nontrivial separable extension of $K$, again contradicting our assumption that $K = l \neq L$. We conclude that $L_0 = K$ and so $t \in \widehat{K^{1/p^\infty}} \setminus K$.

Choose
$a_0=0,a_1,\ldots\in K$ such that the sequence $r_n=|t-a_n^{1/p^n}|$ converges to zero. Then $L$ is the completion
of its subalgebra $\bigcup_n k\{r_0^{-1}t,r_n^{-p^n}(t^{p^n}-a_n)\}$; in particular, $k[t]$ is dense in $L$.
Consider the Banach ring $\calA:=L\widehat{\otimes}_KL$ provided with the tensor product norm $\left\|\bullet\right\|$ and note that
the ideal $J={\rm Ker}(\calA\to L)$ is generated by $T:=1\otimes t-t\otimes 1$.

We claim that $\|T^{p^n}\|\le r_n^{p^n}$. Indeed, since $|t^{p^n}-a_n|=r_n^{p^n}$, we have that
$\|1\otimes t^{p^n}-a_n\|\le r_n^{p^n}$ and $\|t^{p^n}\otimes 1-a_n\|\le r_n^{p^n}$. (Note, for the sake of completeness,
that $T$ is quasi-nilpotent, i.e. its spectral norm vanishes, and hence
$L$ is the uniform completion of $\calA$, i.e. the completion with respect to the spectral seminorm.
This is a topological extension of the classical fact that $T$ is nilpotent and $L$ is the reduction of $\calA$
when $L/K$ is finite and purely inseparable.)

By \cite[Remark~4.3.4(ii)]{temkin}, there is an isomorphism $J/J^2\stackrel\sim\to\widehat{\Omega}_{L/K}$ that takes $T$
to $\widehat{d}_{L/K}(t)$. Thus, we should only show that $T\neq aT^2$ in $\calA$. Assume, to the contrary,
that $T=aT^2$ and set $s=\|a\|$. Then, $\|T\|=\|a^{p^n-1}T^{p^n}\|\le s^{p^n-1}r_n^{p^n}$ for any $n$ and hence
$\|T\|=0$. Thus $L\widehat{\otimes}_KL=L$ and since $L\otimes_KL$ embeds into $L\widehat{\otimes}_KL$ by \cite[3.2.1(4)]{gruson},
we obtain a contradiction.
\end{proof}


\section{Proofs and examples}

We now settle the questions raised in the introduction.

\begin{prop} \label{P:spherically complete isomorphism}
Question~\ref{Q:induced isomorphism} admits an affirmative answer if $K$ is spherically complete.
\end{prop}
\begin{proof}
Let $\tau:K \to K$ be a homomorphism as in Question~\ref{Q:induced isomorphism}.
Suppose by way of contradiction that there exists $x \in K$ with $x \notin \tau(K)$.
Since $K$ is spherically complete, the set of possible valuations of $x - \tau(y)$ for $y \in K$ has a least element. If $y$ realizes this valuation,
then by the matching of value groups, we can find $y' \in K$ such that $\tau(y')$ and
$x - \tau(y)$ have the same valuation; by the matching of residue fields, we can further choose $y'$ such that $(x-\tau(y))/\tau(y')$ maps to 1 in $k$. But then
$x - \tau(y+y')$ has smaller valuation than $x-\tau(y)$, a contradiction.
\end{proof}

\begin{example} \label{exa:no induced isomorphism}
Let $k$ be an analytic field whose absolute value is nondiscrete, and choose a sequence
$x_1, x_2, \ldots \in k^\times$ such that $|x_i|<1$ and $\lim_n|x_1\cdots x_n|>0$.
(For a more concrete example, take $k$ to be a completed algebraic closure of $\CC((t))$ and take
$x_n = t^{2^{-n}}$.) Let $K$ be the completion of $k(t_1, t_2, \dots)$ for the Gauss valuation
(i.e., the valuation of a nonzero polynomial is the maximum valuation of its coefficients);
then $K$ admits a unique valuation-preserving endomorphism $\tau$ fixing $k$ and taking
$t_n$ to $t_n - x_n t_{n+1}$ for each $n$. We will show that the image of $\tau$ does not contain $t_1$, and hence $\tau$ is not an isomorphism.

Suppose to the contrary that there exists $y \in K$ with $\tau(y) = t_1$.
By hypothesis, there exists some $\lambda \in k$ such that $|\lambda| < |x_1 \cdots x_n|$ for all $n$.
We may then choose $y' \in K_0(t_1,\dots,t_n)$ for some positive integer $n$ in such a way that $|y - y'| < |\lambda|$.
Put $y'' = t_1 + x_1 t_2 + \cdots + x_1 \cdots x_n t_{n+1}$;
then $\tau(y'') = t_1 - x_1 \cdots x_{n+1} t_{n+2}$,
so $|y''-y| = |\tau(y''-y)| = |x_1 \cdots x_{n+1}| > |\lambda|$.
Hence $|y'' - y'| = |x_1 \cdots x_{n+1}|$,
but $y''- y'$ equals $x_1 \cdots x_n t_{n+1}$ plus an element of $k(t_1,\dots,t_n)$
and so cannot have valuation less than $|x_1 \cdots x_n|$.
This yields the desired contradiction.
\end{example}

\begin{remark} \label{R:char 0}
Let $k$ be a field of characteristic $0$. For each positive integer $n$, the derivation
$\frac{d}{dt}$ on $k((t))$ extends to the derivation $\partial_n=n^{-1} t^{1/n - 1} \frac{d}{dt^{1/n}}$ on $k((t^{1/n}))$ satisfying
$|\partial_nf|\le|t|^{-1}|f|$ for any $f\in k((t^{1/n}))$. Let $K$ be a completed algebraic closure of $k((t))$;
by Puiseux's theorem,  $K$ is the completion of $\bigcup_{n=1}^\infty \overline{k}((t^{1/n}))$ for $\overline{k}$ the algebraic closure of $k$ in $K$,
so the derivation $\frac{d}{dt}$ extends uniquely to a continuous derivation on $K$.
Consequently, $\widehat{\Omega}_{K/k}$ is generated by $\widehat{d}_{K/k}(x)$ for
any $x \in K - \overline{k}$. For any $k$-linear automorphism $\tau$ of $K$,
let $L$ be the completion of $\tau(K)(t)$ within $K$; taking $x = \tau(t)$ in the previous discussion shows that
$\widehat{\Omega}_{L/\tau(K)} = 0$. By Lemma~\ref{L:primitive}(b) we conclude that
$L/\tau(K)$ is of type 1. Since $\tau(K)$ is algebraically closed, it follows that $L = \tau(K)$ and hence $\tau$ is an isomorphism.
\end{remark}

\begin{proof}[Proof of Theorem~\ref{T:induced isomorphism}]
Choose a sequence $\{d_i\}_{i=1}^\infty$ of positive integers in such a way that:
\begin{enumerate}
\item[(a)]
$d_i$ is not divisible by $p$;
\item[(b)]
$\lim_{i \to \infty} (d_{i+1} - pd_i) = \infty$; and
\item[(c)]
the sequence $\{p^{-i} d_i\}_{i=1}^\infty$
is strictly increasing (for large $i$, this follows from (b))
and bounded.
\end{enumerate}
For a concrete example, take
\[
d_i := 1 + pi + p^2(i-1) + p^3(i-2) + \cdots + p^i.
\]
Choose a sequence $\{c_i\}_{n=1}^\infty$ of elements of $k$ such that each field $\FF_p(c_i)$ is finite, but the field $\FF_p(c_1, c_2, \dots)$ is infinite (in fact any $c_i\neq 0$ will do, but this assumption shortens the argument).
Set
\[
\alpha_n := \sum_{i=0}^n c_i t^{p^{-i} d_i} \in K
\]
and
\[
r_n := \left| \alpha_{n+1} - \alpha_n \right| = |t|^{p^{-n-1} d_{n+1}};
\]
by construction, $\{r_n\}_{n=1}^\infty$ is a strictly decreasing sequence with nonzero limit. Consider the Berkovich affine line $\AAA^1_{k((t))}$ with coordinate $x$ and let $E_n$ be the closed disc of radius $r_n$ centered at $\alpha_n$. The intersection of the $E_n$ does not contain any element of any finite extension of $k((t))$, so it consists of a single point $z$ of type 4. (Otherwise, by \cite[Lemma~10.1, Corollary~11.9]{kedlaya-automata} the generalized power series $x = \sum_{i=0}^\infty c_i t^{p^{-i} d_i}$ would be algebraic over $k(t)$, hence over $\overline{\FF}_p(t)$ because the minimal polynomial must be invariant under coefficientwise automorphisms, hence over $\FF_p(t)$; but then \cite[Corollary~11.9]{kedlaya-automata} would force the $c_i$ to belong to a finite extension of $\FF_p$.) The completed residue field $L=\calH(z)$ of this point is a primitive extension of $k((t))$ topologically generated by $x$, and the conditions $z\in E_n$ mean that $\left| x^{p^n} - \alpha_n^{p^n} \right| = r_n^{p^n}$ for each $n$.
Since $\widehat{d}_{L/k}$ is nonexpanding,
we have $\left\| \widehat{d}_{L/k}(\alpha_n^{p^n}) \right\| \leq r_n^{p^n}$.
Furthermore, in the expression $\alpha_n^{p^n} = \sum_{i=0}^n t^{p^{n-i} d_i}$
only the term $t^{d_n}$ is not a $p$-th power, so
\[
\widehat{d}_{L/k}(\alpha_n^{p^n}) = \widehat{d}_{L/k}(t^{d_n}) = d_n t^{d_n-1} \widehat{d}_{L/k}(t)
\]
and hence (since $d_n$ is not divisible by $p$)
\[
\left\| \widehat{d}_{L/k}(t) \right\| \leq r_n^{p^n} |t|^{1-d_n}
= |t|^{p^{-1} d_{n+1} - d_n + 1}.
\]
Since this holds for all $n$, we conclude that $\left\| \widehat{d}_{L/k}(t) \right\| = 0$.

By the previous paragraph, $\widehat{d}_{L/k}(t) = 0$ and hence
$\widehat{d}_{L/k((x))}(t) = 0$. By Lemma~\ref{L:primitive}, $L/k((x))$ is a primitive extension of type 1;
the inclusion $k((x)) \to L$ thus induces an isomorphism of completed
algebraic closures. That is, $t$ belongs to the completed algebraic closure of $k((x))$,
but $x$ does not belong to the completed algebraic closure of $k((t))$.
If we write $K'$ for a completed algebraic closure of $k((x))$, we then have a strict inclusion
$K \to K'$; composing this with an identification $K' \cong K$ yields the desired endomorphism.
\end{proof}

\begin{proof}[Proof of Theorem~\ref{T:same untilt Cp}]
We use \cite[Theorem~1.5.6]{kedlaya-new-phigamma} as our blanket reference concerning the perfectoid correspondence.
By \cite[Example~1.3.5]{kedlaya-new-phigamma}, there is an algebraic extension of $\QQ_p$ whose completion is perfectoid with tilt isomorphic to the completed perfect closure of $\FF_p((t))$; hence $\CC_p$ is perfectoid and $\CC_p^\flat$ is isomorphic to the completed algebraic closure of $\FF_p((t))$.
By Theorem~\ref{T:induced isomorphism}, there exists an endomorphism $\tau: \CC_p^\flat \to \CC_p^\flat$ which is not surjective; this corresponds to a morphism $\CC_p \to K$ of perfectoid fields which is not surjective either. In particular, the integral closure of $\QQ_p$ in $K$ is not dense, so $K$ cannot admit any isomorphism to $\CC_p$ in the category of topological fields.
\end{proof}

\end{document}